\documentclass{amsart}

\usepackage{enumerate}
\usepackage{amsmath}
\usepackage{amssymb,latexsym}
\usepackage{amsthm}
\usepackage{color}
\usepackage{enumitem}
\usepackage{graphicx}
\usepackage{amscd}
\usepackage{hyperref}

\def\cV{\mathcal V}
\newcommand{\F}{{\mathbb F}}
\newcommand{\cC}{{\mathcal C}}

\newcommand{\K}{{\mathbb K}}

\newcommand{\Fix}{\mathrm Fix}
\newcommand{\la}{\langle}
\newcommand{\ra}{\rangle}

\renewcommand{\mod}{\hbox{{\rm mod}\,}}

\def \bsigma{\boldsymbol{\sigma}}

\newtheorem{theorem}{Theorem}[section]
\newtheorem{lemma}[theorem]{Lemma}
\newtheorem{corollary}[theorem]{Corollary}

\DeclareMathOperator{\PG}{{PG}}

\DeclareMathOperator{\Aut}{{Aut}}

\DeclareMathOperator{\Gal}{Gal}

\newcommand\spanset[1]{\ensuremath \langle #1 \rangle}

\theoremstyle{definition}%% Hans
\newtheorem{definition}[theorem]{Definition}
\newtheorem{example}[theorem]{Example}

\definecolor{byzantium}{rgb}{0.44, 0.16, 0.39}
\def\vale#1 {\fbox {\footnote {\ }}\ \footnotetext { From Vale: {\color{byzantium}#1}}}

\def\nico#1 {\fbox {\footnote {\ }}\ \footnotetext { From Nico: {\color{red}#1}}}

\def\giovanni#1 {\fbox {\footnote {\ }}\ \footnotetext { From Giovanni: {\color{blue}#1}}}

\title[$(d,\bsigma)$-Veronese variety and some applications]{$(d,\bsigma)$-Veronese variety and some  applications}

\author{N. Durante\textsuperscript{\,*}}
\address{\textsuperscript{*}Dipartimento di Matematica e Applicazioni ``Renato Caccioppoli''\\
	Universit\`a degli Studi di Napoli ``Federico II''\\
	Via Vicinale Cupa Cintia, 26, 80126 Napoli, Italy}
\email{\{giovanni.longobardi,ndurante\}@unina.it}
\author{G. Longobardi\textsuperscript{\,*}}

\author{V. Pepe\textsuperscript{\,$\dagger$}}
\address{\textsuperscript{$\dagger$}Dipartimento di Scienze di Base ed Applicate per l'Ingegneria\\
'Sapienza' Università di Roma\\
Via Antonio Scarpa, 14\\
00161 Roma, Italy}
\email{valepepe@sbai.uniroma1.it}

\begin{document}
\maketitle
\begin{abstract}
Let $\K$ be the Galois field $\F_{q^t}$ of order $q^t, q=p^e, p$ a prime, $A=\Aut(\K)$ be the automorphism group of $\K$ and  $\boldsymbol{\sigma}=(\sigma_0,\ldots, \sigma_{d-1}) \in A^d$, $d \geq 1$. In this paper  the following generalization of the Veronese map is studied:
$$
\nu_{d,\boldsymbol{\sigma}}   :  \spanset{v} \in \PG(n-1,\K) \longrightarrow \spanset{v^{\sigma_0} \otimes v^{\sigma_1} \otimes \cdots \otimes v^{\sigma_{d-1}}}\in \PG (n^d-1,\K ).
$$
Its image will be called  the $(d,\bsigma)$-\textit{Veronese variety} $\cV_{d,\bsigma}$.
 For $d=t$, $\sigma$ a generator of $\mathrm{Gal}(\F_{q^t}|\F_{q})$ and $\bsigma=(1,\sigma,\sigma^2,\ldots,\sigma^{t-1})$,  the $(t,\bsigma)$-Veronese variety $\cV_{t,\bsigma}$
 is the variety studied in \cite{spread,normal,pepe11}. Such a variety is the Grassmann embedding of the Desarguesian spread of $\PG(nt-1,\F_q)$ and it has been used to construct codes \cite{giuzzipepe13} and (partial) ovoids of quadrics, see \cite{normal,pepe12}.
Here, we will show that  $\cV_{d,\bsigma}$  is the Grassmann embedding of a normal rational scroll  and any $d+1$ points of it are linearly independent. We give a characterization of $d+2$ linearly dependent points of $\cV_{d,\bsigma}$ and for some choices of parameters, $\cV_{p,\bsigma}$ is the normal rational curve; for $p=2$, it can be the Segre's arc of $\PG(3,q^t)$; for $p=3$ $\cV_{p,\bsigma}$ can be also a $|\cV_{p,\bsigma}|$-track of $\PG(5,q^t)$. Finally, investigate the link between such points sets and a linear code $\cC_{d,\bsigma}$ that can be associated to the variety, obtaining examples of MDS and almost MDS codes.
\end{abstract}

\section{Introduction}

Let $V=V(n,\K)$ be an $n$-dimensional vector space over a field $\K$, we will denote by $\PG(V)$ as well as $\PG(n-1,\K)$ the projective space induced by it.

We refer to \cite{Harris} for the definition of dimension, degree, smoothness and tangent space of an algebraic variety and for the methods and techniques used to study classical varieties.

The Veronese variety $\mathcal{V}_d$ of degree $d$ and dimension $n-1$ is a classical algebraic variety widely studied over fields of any characteristic \cite{Harris,GGG} and it is the image of the Veronese map
\begin{equation*}
\nu_d: (x_0,x_1,\ldots,x_{n-1}) \in \PG(n-1,\K) \longrightarrow (\ldots,X_I,\ldots) \in \PG\left(\binom{n+d-1}{d}-1,\K\right)
\end{equation*}

where $X_I$ ranges over all the possible monomials of degree $d$ in $x_0,x_1,\ldots,x_{n-1}$.

The Veronese map can be defined also by
\begin{equation*}
\nu_{d}   :  \spanset{v} \in \PG(n-1,\K) \longrightarrow \spanset{v \otimes v \otimes \cdots \otimes v}\in \PG\left(\binom{n+d-1}{d}-1,\K\right).
\end{equation*}

Now, let $V_i$ be  $n_i$-dimensional vector spaces over the field $\K$, $i=0,1,\ldots,d-1$. A \textit{Segre variety of type} $(n_0,n_1,\ldots,n_{d-1})$ in $\PG\bigl (\bigotimes_{i=0}^{d-1} V_i \bigr )$ is the set

\begin{equation}\label{tensor-Segre}
\Sigma_{n_0-1,n_1-1,\ldots,n_{d-1}-1}= \bigl \{\spanset{v_0 \otimes v_1 \otimes \cdots \otimes v_{d-1}} \,\,|\,\, v_i \in V_i \setminus \{0\}, \, i=0,1,\ldots,d-1 \bigr\}
\end{equation}

If $n_0=\ldots=n_{d-1}=n$, we write $\Sigma_{(n-1)^d}$ instead of $\Sigma_{n-1,n-1,\ldots,n-1}$. Then it is clear that $\cV_{d}$ turns out to be a linear section of the Segre variety product of $\PG(n-1,\K)$ for itself $d$ times.

If $\mathbf{\zeta}$ is a collineation of $\PG \bigl (V^{\otimes d} \bigl )$ fixing $\Sigma_{(n-1)^d}$, then there exist $\zeta_i, i=0,1,\ldots,d-1$ semilinear maps of $\PG(V)$, with the same companion field automorphism, and a permutation $\tau$ on $\{0,1,\ldots,d-1\}$ such that

$$
\langle v_0 \otimes v_1 \otimes \cdots \otimes v_{d-1} \rangle ^{\mathbf{\zeta}}= \langle v_{\tau(0)}^{\zeta_0}\otimes v_{\tau(1)}^{\zeta_1}\otimes \cdots \otimes v_{\tau(d-1)}^{\zeta_{d-1}} \rangle,
$$

for a proof of this in positive characteristic see \cite{westwick}.

Let  $\mathcal{L}_h$ be the set of all projective subspaces of dimension $h$ of $\PG(n-1,\K)$, and consider
\begin{equation*}
g_{n,h}: \spanset{v_0, v_1,v_2,\ldots,v_h} \in \mathcal{L}_h \longrightarrow \spanset{v_0 \wedge v_1 \wedge v_2 \wedge \cdots \wedge v_h} \in \PG \bigl ( \bigwedge\hspace{-0,15cm}\,^{h+1}V \bigr).
\end{equation*}
where $\wedge$ is the wedge product and $\bigwedge^{h+1}V$ the $(h+1)$-th exterior power of $V$. This map is called \textit{Grassmann embedding} and its image $\mathcal{G}_{n,h}(V)$ is called \textit{Grassmanian of subspaces of dimension} $h$ of $\PG(V)$. It is well-known that $\mathcal{G}_{n,h}(V)$ is an algebraic variety which is the complete intersection of certain quadrics, see \cite{Harris}.\\

\noindent Let $\K$ be the Galois field $\F_{q^t}$ of order $q^t$, $A=\Aut(\K)$ be the automorphism group of $\K$ and  $\boldsymbol{\sigma}=(\sigma_0,\ldots, \sigma_{d-1}) \in A^d$, $d \geq 1$. The aim of this paper is to study the following generalization of the Veronese map
\begin{equation*}
\nu_{d,\boldsymbol{\sigma}}   :  \spanset{v} \in \PG(n-1,\K) \longrightarrow \spanset{v^{\sigma_0} \otimes v^{\sigma_1} \otimes \cdots \otimes v^{\sigma_{d-1}}}\in \PG (n^d-1,\K )
\end{equation*}
and some properties of its image  that will be called here the $(d,\bsigma)$-\textit{Veronese variety} $\cV_{d,\bsigma}$.
 For $d=t$, $\sigma$ a generator of $\mathrm{Gal}(\F_{q^t}|\F_{q})$ and $\bsigma=(1,\sigma,\sigma^2,\ldots,\sigma^{t-1})$,  the $(t,\bsigma)$-Veronese variety $\cV_{t,\bsigma}$
 is the variety studied in \cite{spread,normal,pepe11}. Such a variety is the Grassmann embedding of the Desarguesian spread of $\PG(nt-1,\F_q)$ and it has been used to construct codes \cite{giuzzipepe15} and (partial) ovoids of quadrics, see \cite{normal,pepe12}.\\
A $[\nu,\kappa]$-\textit{linear code} $\cC$ is a subspace of the vector space $\F_{q}^\nu$ of dimension $\kappa$. The \textit{weight} of a codeword is the number of its entries that are nonzero and the \textit{Hamming distance} between two codewords is the number of entries in which they differ.  The distance $\delta$ of a linear code is the minimum distance between distinct codewords and it is equals to the minimum weight. A linear code of length $\nu$, dimension $\kappa$, and minimum distance $\delta$ is called a $[\nu,\kappa,\delta]$-code.  A matrix $H$ of order $(\nu-\kappa) \times \nu$ such that
\begin{equation*}
\mathbf{x}H^T=\mathbf{0} \hspace{0.5cm} \textnormal{for all} \,\,\, \mathbf{x} \in \cC
\end{equation*}
is called  a \textit{parity check matrix} for $\cC$. The minimum weight, and hence the minimum distance, of $\cC$ is at least $w$ if and only if any $w-1$ columns of $H$ are linearly independent  \cite[Theorem 10, p. 33]{MC}. Each linear $[\nu,\kappa,\delta]$-code $\cC$ satisfies the following inequality
$$ \delta \leq \nu-\kappa+1,$$
called \textit{Singleton bound}. If  $\delta=\nu-\kappa+1$, $\cC$  is called \textit{maximum distance separable} or \textit{MDS},  while if $\delta=\nu-\kappa$ the code is called \textit{almost MDS}.
These can be related to some subsets of points in the projective space. More precisely,  $\cC$ is a $[\nu,\kappa,\delta]$-linear code if and only if the columns of its parity check matrix $H$ can be seen as $\nu$ points in $\PG(\nu-\kappa-1,q)$ each $\delta-1$ of which are in general position, \cite[Theorem 1]{deboer}.
Then, the existence of MDS or almost MDS linear codes is equivalent to the existence of arcs  or tracks in projective spaces, respectively.
\begin{definition}
A $k$-arc is a set of $k$ points in $\PG(r,q)$ such that $r+1$ of them are in general position.
An $\ell$-\textit{track} is a set of $\ell$ points in $\PG(r,q)$ such that every $r$ of them are in general position.
\end{definition}
Here, we study the variety $\cV_{d,\bsigma}$ and we will prove that it is the Grassmann embedding of a normal rational scroll and that any $d+1$ points of it are in general position, i.e. any $d+1$ points of $\cV_{d,\bsigma}$ are linearly independent. Moreover, we give a characterization of $d+2$ linearly dependent points of this variety and investigate how such a property is interesting for a linear code $\cC_{d,\bsigma}$ that can be associated to the variety.

\section{The variety $\cV_{d,\bsigma}$}

Let $V=V(n,\K)$ be an $n$-dimensional vector space over the field $\K$ and $\PG(V)=\PG(n-1,\K)$ be the induced projective space.
 In particular, if $\K$ is the Galois field of order $q^t$, we will denote the projective space by $\PG(n-1,q^t)$

Let $A=\Aut(\K)$ be the automorphism group of $\K$ and  $\boldsymbol{\sigma}=(\sigma_0,\ldots, \sigma_{d-1}) \in A^d$, $d \geq 1$, and define the map
\begin{equation}\label{veronesemap}
\nu_{d,\boldsymbol{\sigma}}   :  \spanset{v} \in \PG(V) \longrightarrow \spanset{v^{\sigma_0} \otimes v^{\sigma_1} \otimes \cdots \otimes v^{\sigma_{d-1}}}\in \PG \bigl (V^{\otimes d} \bigl ).
\end{equation}

Up to the action of the group P$\Gamma$L(V), we may assume that $\sigma_0=1$. It is clear that the map $\nu_{d,\bsigma}$ is an injection of $\PG(V)$ into $\PG(V^{\otimes d})$ by the injectivity of the Segre map.

We will call $\nu_{d,\boldsymbol{\sigma}}$ the $(d,\boldsymbol{\sigma})$-\textit{Veronese embedding} and, as defined before, its image $\cV_{d,\boldsymbol{\sigma}}$ the $(d,\boldsymbol{\sigma})$-Veronese variety. Then $\cV_{d,\bsigma}$ is a rational variety of dimension $n-1$ in $\PG(N-1,\K)$, $N=n^d$ and it has as many points as $\PG(n-1,q^t)$, see \cite{GGG, Harris}.\\

As a consequence of \cite[Theorem 3.5 and 3.8]{westwick}, one gets the following 
\begin{theorem}
Let $\zeta$ be a collineation of $\PG \bigl (V^{\otimes d} \bigl )$. Then $\zeta$ fixes $\cV_{d,\boldsymbol{\sigma}}$ if and only if
\begin{equation*}
 \langle v\otimes v^{\sigma_1}\otimes \cdots \otimes v^{\sigma_{d-1}} \rangle ^{\mathbf{\zeta}}=  \langle v^{\zeta_0}\otimes v^{ \zeta_0\sigma_1}\otimes \cdots \otimes v^{ \zeta_{0}\sigma_{d-1}} \rangle \quad \textnormal{for any $v \in V \setminus \{0\}$}
\end{equation*}
where $\zeta_0$, is a bijective semilinear map of $V$.
\end{theorem}
Note that applying the map
$$
\langle v\otimes v^{\sigma_1}\otimes \cdots \otimes v^{\sigma_{d-1}}\rangle ^{\mathbf{\zeta} }= \langle v^{\zeta_0}\otimes v^{\zeta_1\sigma_1}\otimes \cdots \otimes v^{\zeta_{d-1}\sigma_{d-1}} \rangle$$
where  $\zeta_i$ is a bijective semilinear map, we get a subvariety of $\Sigma_{(n-1)^d}$ projectively equivalent to $\cV_{d,\boldsymbol{\sigma}}$.

Although many of the results also hold in the case of a general field, from now on it will be assumed, that $\K$ is the Galois field $\F_{q^t}$ of $q^t$ elements and  $\boldsymbol{\sigma}=(\sigma_0,\sigma_1,\ldots,\sigma_{d-1}) \in G^d$ with $G=\Gal(\F_{q^t} |\, \F_q)$.
Moreover, since any element $\sigma_i \in G$ is a map of the type $\sigma_i : x \mapsto x^{q^{h_i}}$ with $0 \leq h_i < t$ and  $0 \leq i \leq d-1$, hereafter we will suppose that
\begin{equation*}
\bsigma=(\underbrace{\sigma_0,\ldots,\sigma_0}_{d_0\, \textnormal{times}},\underbrace{\sigma_1,\ldots,\sigma_1}_{d_1\,\textnormal{times}},\ldots,\underbrace{\sigma_{m},\ldots,\sigma_{m}}_{d_m\, \textnormal{times}})
\end{equation*}
where $0=h_0<h_1<\ldots<h_m<t$ and we will consider the vector $d_{\bsigma}=(d_0,d_1,\ldots,d_m)$ where $d_j$ is the occurrence of $\sigma_j$  in $\bsigma$, $0 \leq j \leq m$. Clearly $d_0+d_1+\ldots+d_m=d$. If $\bsigma \in G^d$, the integer
\begin{equation}\label{sigmanorm}
	|\bsigma|= \sum_{i=0}^{d-1}q^{h_i}=\sum_{i=0}^{m}d_iq^{h_i}.
\end{equation}
will be called $norm$ of $\bsigma$.

Since we consider the ring of polynomials $\F_{q^t}[x_0,x_1,\ldots,x_{n-1}]$ actually as the quotient $\F_{q^t}[x_0,x_1,\ldots,x_{n-1}]/(x_0^{q^t}-x_0,x_1^{q^t}-x_1,\ldots,x_{n-1}^{q^t}-x_{n-1})$,\textbf{ from now on we assume } $|\bsigma|< q^t$, so that distinct polynomials will be distinct functions over $\F_{q^t}$.
By injectivity of map in \eqref{veronesemap}, it is clear that $(d,\bsigma)$-Veronese variety $\cV_{d,\bsigma}$ has as many points as $\PG(n-1,q^t)$. \\

Let $\{e_i \, \,|\,\,i=0,1,\ldots,nd-1\}$ be the canonical basis of $V(nd,\F_{q^t})=V(nd,q^t)$ and let $\Pi\cong \PG(n-1,q^t)$ be the subspace of $\PG(nd-1,q^t)$ spanned by $\{\spanset{e_i}\,\,|\,\,0 \leq i \leq n-1\}$. Let $\phi$ be the collineation of $\PG(nd-1,q^t)$ such that
 $$\la e_i \ra \mapsto \la e_{i+n} \ra,$$
where the subscripts are taken modulo $nd$. As done in \cite[Section 4]{giuzzipepe15}, for any $\langle v_i \rangle \in \Pi^{\phi^i}$, we can identify
$v_0\otimes v_1\otimes v_2\otimes \cdots \otimes v_{d-1}$ with $v_0\wedge v_1\wedge v_{2}\wedge \cdots \wedge v_{d-1}$.
Therefore,  $\cV_{d,\boldsymbol{\sigma}}$ is the  Grassman embedding of the $d$-fold normal rational scroll $$S^{\boldsymbol{\sigma}}_{n-1,n-1,\ldots,n-1}=\bigl \{\langle P^{\sigma_0 },P^{\phi\sigma_1},P^{\phi^2  \sigma_2},\ldots,P^{\phi^{d-1}  \sigma_{d-1}}\bigr \rangle\, |\,P \in \Pi \}$$ of $\PG(nd-1,q^t)$,
see \cite[Ch.8]{Harris} for a definition of normal rational scroll.

\begin{example} \label{example1}
	Let $\boldsymbol{\sigma}=\boldsymbol{1}$, the identity of the product group $G^d$, the $(d,\bsigma)$-Veronese variety $\cV_{d,\bsigma}$ is the classical Veronese variety of degree $d$ and $\cV_{d,\bsigma} \subset \PG(N-1,q^t)$ with $N=\binom{n+d-1}{d}$. In this case,  $\cV_{d,\bsigma}$ is the Grassmann embedding of $S_{n-1,n-1,\ldots,n-1}=\{\langle P,P^{\phi},P^{\phi^2},\ldots,P^{\phi^{d-1}}\rangle\, |\,P \in \Pi \}$, i.e. the Segre variety $\Sigma_{n-1,d-1}$ of $\PG(nd-1,q^t)$, see again \cite[Ch.8]{Harris}.
\end{example}

\begin{example}\label{example2}
	Let	$\sigma$ be a generator of $\Gal(\F_{q^t}|\F_q)$ and $\bsigma=(1,\sigma,\ldots,\sigma^{t-1})$, then we get the algebraic variety introduced in \cite{spread,normal,pepe11} and we will refer to it as the SLP-\textit{variety} $\cV_{t,\sigma}$. Let $\hat{\sigma}$ be  the semi-linear collineation $\phi\circ \sigma$ of $\PG(nt-1,q^t)$ of order $t$. Then the set of points fixed by $\hat{\sigma}$, $\Fix(\hat{\sigma})\subset \PG(nt-1,q^t)$, is a subgeometry  isomorphic to $\PG(nt-1,q)$  and a subspace of $\PG(nt-1,q^t)$ intersects the subgeometry in a subspace of the same dimension if and only if it is set-wise fixed by $\hat{\sigma}$ (see \cite[Section 3]{normal}). In this case
	$$S^{\bsigma}_{n-1,n-1,\ldots,n-1}=\bigl \{\langle P,P^{\hat{\sigma}},P^{\hat{\sigma}^2},\ldots,P^{\hat{\sigma}^{t-1}}\rangle\, |\,P \in \Pi \bigr\},$$ and hence its $(t-1)$-spaces  are set-wise fixed by $\hat{\sigma}$. Also, $S_{n-1,n-1,\ldots,n-1}\cap \Fix(\hat{\sigma})$ is the Desarguesian $(t-1)$-spread of $\PG(nt-1,q)=\Fix(\hat{\sigma})\subset \PG(nt-1,q^t)$. Therefore, $\cV_{t,\sigma}$ is the Grassmann embedding of the Desarguesian spread of $\PG(nt-1,q)$. In this case, in fact, $\cV_{t,\sigma}$ turns out to be a variety of the subgeometry $\PG(n^t-1,q)\subset \PG(n^t-1,q^t)$ point-wise fixed by the semi-linear collineation of order $t$ of $\PG(n^t-1,q^t)$ induced by  $\hat{\sigma}$:
\begin{equation*}
v_0\otimes v_1\otimes \cdots \otimes v_{t-1}\mapsto v_{t-1}^{\hat{\sigma}} \otimes v_0^{\hat{\sigma}}\otimes \cdots \otimes v_{t-2}^{\hat{\sigma}}.
\end{equation*}
\end{example}
\vspace{0.8cm}
\noindent By \eqref{veronesemap}, a point of $\PG(n-1,q^t)$ with homogeneous coordinates $(x_0,x_1,\ldots,x_{n-1})$ is mapped by $\nu_{d,\bsigma}$ into a point of coordinates
$$(...,\displaystyle \prod_{j=0}^{m}X_{I_j}^{\sigma_j},...)$$
where $X_{I_j}$ is a monomial of degree $d_j$ in the variables $x_0,x_1,\ldots,x_{n-1}$. Hence,  the $(d,\bsigma)$-Veronese variety $\cV_{d,\bsigma}$ is contained in a projective space of vector space dimension
\begin{equation} \label{N}
 N=N_0N_1\cdots N_m, \quad  N_j=\binom{n+d_j-1}{d_j},\quad j=0,1,\ldots,m.
\end{equation}
Let $\sigma_0,\sigma_1,\ldots \sigma_m$ distinct automorphisms in $\bsigma$ and $d_{\bsigma}$ the vector of their occurences and suppose that $d_i\sigma_i\neq d_j\sigma_j$ for all $i,j=0,1,\ldots,m$ distinct, then we get exactly $N=n^d$ distinct monomials of type $\displaystyle \prod_{j=0}^{m}X_{I_j}^{\sigma_j}$. This is not the case anymore if $d_i\sigma_i=d_j\sigma_j$ for some $i\neq j$. For example, if $q=2$, $\bsigma=(1,1,2)$, then $d_0=2,d_1=1$ and hence $d_0=d_1\sigma_1$. Then
$$(x_0,x_1)\otimes (x_0,x_1)\otimes (x_0^2,x_1^2)=(x_0^4,x_0^2x_1^2,x_0^3x_1,x_0x_1^3,x_0^3x_1,x_0x_1^3,x_0^2x_1^2,x_1^4),$$
and we get 5 distinct monomials and $\cV_{3,\bsigma}$ is in fact contained in a projective space of vector space dimension less than $N=6$.

Recall that an $r$-\textit{hypersurface} of $\PG(n-1,q^t)$ is a variety such that its points have coordinates vanish an $r$-form of $\F_{q^t}[X_0,\ldots,X_{n-1}].$ If $r=2$, an $r$-hypersurface is called  \textit{quadric}. In \cite{tallini}, it is shown a lower bound on the degree of an $r$-hypersurface  $\mathcal{D}$  of $\PG(n-1,q^t)$ after which $\mathcal{D}$ could contain all points of the projective space. More precisely,
\begin{theorem}\label{Tallini}
	\cite{tallini}	If an $r$-hypersurface $\mathcal{D}$ of $\PG(n-1,q^t)$ contains all the points of the space, then $r \geq q^t+1$.
\end{theorem}

Let $I$ be a multi-index of the form $I=I_0I_1\cdots I_m$, where $I_j$ is a multi-index corresponding to a monomial in $x_0,x_1,\ldots,x_{n-1}$ of degree $d_j$.  Once we have labelled the coordinates of $\PG(N-1,q^t)$ according to the multi-index $I$, we can define a natural linear map  $\psi$ that sends  the hyperplane of $\PG(N-1,q^t)$ of equation $\displaystyle\sum_{I}a_Iz_I=0$ to the $\bsigma$-\textit{hypersurface} of equation
\begin{equation*}
\sum_{I}a_I \displaystyle \prod_{j=0}^{m}X_{I_j}^{\sigma_j}=0.
\end{equation*}

Then, by Theorem \ref{Tallini}, we get the following result.

\begin{theorem}\label{hyperplane}

Let $\bsigma \in G^d$ with $d_{\bsigma}=(d_0,d_1,\ldots,d_m)$, $| \boldsymbol{\sigma}|<q^t$. The $(d,\bsigma)$-Veronese variety $\cV_{d,\bsigma}$ is not contained in any hyperplane of $\PG(N-1,q^t)$ with $N=N_0 N_1\cdots N_m$
and \[N_j=\binom{n+d_j-1}{d_j}, \quad j=0,1,\ldots,m.\]
\end{theorem}

In the following, we generalize some results proved in \cite[Section 2]{giuzzipepe13} for the SLP-variety.

\begin{theorem}\label{indep}
	Let $\Pi_0,\Pi_1,\ldots,\Pi_{d-1}$ be proper subspaces of $\PG(n-1, q^t)$ and suppose
	that $P \in \PG(n-1, q^t)$ is not contained in any of them. Then, $P^{\nu_{d,\bsigma}}$ is not contained in $\langle \Pi_0^{\nu_{d,\bsigma}},\Pi_1^{\nu_{d,\bsigma}},\ldots,\Pi_{d-1}^{\nu_{d,\bsigma}} \rangle$.
\end{theorem}

\begin{proof}
	
Recall that the dual space of $V(n^d,q^t)$, denoted by $V(n^d,q^t)^*$, is spanned by the simple tensors $l_0^ *\otimes l_1^*\otimes \cdots \otimes l_{d-1}^*$, with $l_i^* \in V(n,q^t)^*$, and $l_0^*\otimes l_1^*\otimes \cdots \otimes l_{d-1}^*$ evaluated in $u_0\otimes u_1 \otimes \cdots \otimes u_{d-1}$ is $l_0^*(u_0)l_1^*(u_1)\cdots l_{d-1}^*(u_{d-1}) \in \mathbb{F}_{q^t}$.

	For every  $i \in \{0,1,\ldots,d-1\}$, take an $l_i^* \in  V(n,q^t)^*$ such that $l_i^*$ vanishes on $\Pi_i^{\sigma_i}$ and not in $P^{\sigma_i}$. Then the hyperplane defined by $ l_0^*\otimes l_1^*\otimes \cdots \otimes l_{d-1}^*$ contains the points of $\Pi_j^{\nu_{d,\bsigma}}$ $\forall\, j=0,1,\ldots, d-1$ and it does not contain the point $P^{\nu_{d,\bsigma}}$.
\end{proof}

\begin{corollary}\label{indep2}
	Any $d+1$ points of $\mathcal{V}_{d,\bsigma}$, $d \geq 2$, are in general position.
\end{corollary}
\begin{proof}
	It is enough to  take the $\Pi_i$'s of dimension 0.
\end{proof}

\begin{corollary}\label{dep}
	A set of $d+2$ linearly dependent points of  $\mathcal{V}_{d,\bsigma}$ is the $(d,\bsigma)$-Veronese embedding of points contained in a line of $\PG(n-1,q^t)$.
\end{corollary}
\begin{proof}
	The statement needs to be proved for $n>2$. Let $P_0,P_1,\ldots,P_d,P_{d+1}$ be $d+2$ points whose embedding is linearly dependent. Let $\Pi_i:=P_i$ for $i=2,\ldots,d+1$ and let $\Pi_1=\langle P_0,P_{1} \rangle$. Suppose that $P_i\notin \Pi_1$, with $i=2,\ldots,d+1$, then by Theorem \ref{indep},
$$P_i^{\nu_{d,\bsigma}} \notin \langle \Pi_1^{\nu_{d,\bsigma}},\Pi_2^{\nu_{d,\bsigma}},\ldots,\Pi_{i-1}^{\nu_{d,\bsigma}},\Pi_{i+1}^{\nu_{d,\bsigma}},\ldots,\Pi_{d+1}^{\nu_{d,\bsigma}} \rangle,$$ but by hypothesis
$$P_i^{\nu_{d,\bsigma}} \in \langle P_0^{\nu_{d,\bsigma}},P_1^{\nu_{d,\bsigma}},\ldots,P_{i-1}^{\nu_{d,\bsigma}},P_{i+1}^{\nu_{d,\bsigma}},\ldots, P_{d+1}^{\nu_{d,\bsigma}}\rangle \subset  \langle \Pi_1^{\nu_{d,\bsigma}},\Pi_2^{\nu_{d,\bsigma}},\ldots,\Pi_{i-1}^{\nu_{d,\bsigma}},\Pi_{i+1}^{\nu_{d,\bsigma}},\ldots,\Pi_{d+1}^{\nu_{d,\bsigma}} \rangle,$$
a contradiction.
\end{proof}

In order to prove the next Corollary, we need the following

\begin{lemma}\cite{lunardon84}\label{regulus}
	Let $d<|\mathbb{K}|$. Let $S$ be a set of $d+2$ subspaces of $\PG(2d-1,\mathbb{K})$ of dimension $d-1$,  pairwise disjoint, linearly dependent as points of the Grassmannian and such that any $d+1$ elements of $S$ are linearly independent. Then a line intersecting 3 elements of $S$ intersects all of them.
\end{lemma}

Since we have assumed $|\bsigma|<q^t$, Lemma \ref{regulus} always applies to $\mathcal{V}_{d,\bsigma}$.

\begin{corollary}\label{line}
A set of $d+2$ linearly dependent points of  $\mathcal{V}_{d,\bsigma}$ is  the Grassmann embedding of $(d-1)$-subspaces of the normal rational scroll $S_{1,1,\ldots,1}\subset \PG(2d-1,q^t)$ such that a line intersecting 3 of them must intersect all of them.
\end{corollary}
\begin{proof}
	By Corollary \ref{dep}, a set $\{P_0^{\nu_{d,\bsigma}},P_1^{\nu_{d,\bsigma}},\ldots,P_{d+1}^{\nu_{d,\bsigma}}\}$ of $d+2$ linearly dependent points of  $\mathcal{V}_{d,\bsigma}$ is such that $P_0,P_1,\ldots,P_{d+1}$ are contained in the same line, hence $\{P_0^{\nu_{d,\bsigma}},P_1^{d,\nu_{\bsigma}},\ldots,P_{d+1}^{\nu_{d,\bsigma}}\}$ is contained in a variety $\mathcal{V}_{d,\bsigma}$ of dimension 1.\\ Hence, $\{P_0^{\nu_{d,\bsigma}},P_1^{\nu_{d,\bsigma}},\ldots,P_{d+1}^{\nu_{d,\bsigma}}\}$ is the Grassmann embedding of the $(d-1)$-subspaces of the normal rational scroll $S_{1,1,\ldots,1}\subset \PG(2d-1,q^t)$. Then the result follows from Corollary \ref{indep2} and Lemma \ref{regulus} .

\end{proof}

\begin{theorem}\label{r+2} A set of $d+2$ linearly dependent points of  $\mathcal{V}_{d,\bsigma}$ is the $\bsigma$-Veronese embedding of points on a subline $\cong \PG(1,q')$, where $\mathbb{F}_{q'}$ is the largest subfield of $\mathbb{F}_{q^t}$ fixed by $\sigma_i$ in $\bsigma$.
\end{theorem}
\begin{proof}
	Let $ \langle u_i\otimes u_i^{\sigma_1}\otimes \cdots \otimes u_i^{\sigma_{d-1}} \rangle$, $i=0,1,\ldots,d+1$ be $d+2$ linearly dependent points of  $\mathcal{V}_{d,\bsigma}$, and by Corollary \ref{dep}, we can assume $\mathcal{V}_{d,\bsigma}$ to be of dimension 1. Then, embed $\PG(1,q^t)$ as the subspace of $\PG(2d-1,q^t)$ spanned by $ \langle e_0 \rangle , \langle e_1 \rangle $, say $\Pi$, and hence we can write
	
	$$u_i\otimes u_i^{\sigma_1}\otimes \cdots \otimes u_i^{\sigma_{d-1}}=u_i\wedge u_i^{\phi\sigma_1}\wedge   \cdots \wedge u_i^{\phi^{d-1}\sigma_{d-1}}.$$
	
	We stress out that $\phi^j$ and $\sigma_j$ commute and that the vectors $u_i$'s are pairwise not proportional. Let $S_i:=\langle u_i,u_i^{\phi \sigma_1},\ldots,u_i^{\phi^{d-1} \sigma_{d-1}} \rangle$, for all $i=0,1,\ldots,d+1$, so we observe that $S_i\cap S_j= \emptyset \,\, \forall \,\, i \neq j$. Then take a point $P \in S_0$ such that $P \notin \langle  \Pi^{\phi^h}, h \neq j \rangle$ for any fixed $j \in \{0,1,\ldots,d-1\}$. The subspace $\langle P,S_1 \rangle$ intersects $S_2$ in a point, say $R$. Let $\ell$ be the line spanned by $P$ and $R$. Then $\ell$ has non empty intersection with $S_1$ as well. Hence, by Corollary \ref{line}, $\ell$ has non empty intersection with all the $S_i$'s. By the choice of $P$, the line $\ell$ is not contained in  any $\langle  \Pi^{\phi^h}, h \neq j \rangle$ for a fixed $j \in \{0,1,\ldots,d-1\}$. If $\ell$ intersects $\langle  \Pi^{\phi^h}, h \neq j \rangle$ for some $j \in \{0,1,\ldots,d-1\}$, then it would be projected to a unique point of $\Pi^{\phi^j}$ from $\langle  \Pi^{\phi^h}, h \neq j \rangle$. Since $u_i\neq u_h$ $\forall i\neq h$, then $u_i^{\sigma_j}\neq u_h^{\sigma_j}$ $\forall i\neq h$ and $\ell$ can be projected on a unique point only if $\ell \cap S_i$ is in $\langle  \Pi^{\phi^h}, h \neq j \rangle$ for all the $S_i's$ except one, a contradiction. Indeed, the point $\ell \cap S_i= \langle \lambda_0u_i+\lambda_1u_i^{\phi\sigma_1}+\ldots+\lambda_{d-1}u_i^{\phi^{d-1}\sigma_{d-1}} \rangle$ and the projection of $\ell\cap S_i$ over $\Pi^{\phi^j}$ is the point $\langle u_i^{\phi^j\sigma_j} \rangle$, so $h_j$ cannot be zero. Therefore, $\ell \cap\langle  \Pi^{\phi^h}, h \neq j \rangle=\emptyset $ for any fixed $j \in \{0,1,\ldots,r-1\}$.
 Hence the projection of $\ell$ on a $\Pi^{\phi^j}$ is an isomorphism of lines, say $p_j$ and $(\ell \cap S_i)^{p_j}= \langle u_i^{\phi^j\sigma_j}\rangle $. By $(\ell \cap S_i)^{p_j\phi^{-j}}=(\ell \cap S_i)^{p_0\sigma_j}$ we get that $(\ell \cap S_i)^{p_0}$ is fixed by the semi-linear collineation $\sigma_j\phi^jp^{-1}_jp_0$. If a semi-linear collineation of $\Pi \cong\PG(1,q^t)$ fixes at least 3 points, then it fixes a subline $\cong \PG(1,q')$, where $\mathbb{F}_{q'}$ is the subfield of $\mathbb{F}_{q^t}$ fixed by $\sigma_j$. This is true for all $\sigma_j$ in $\bsigma$.
\end{proof}

Finally, since the algebraic variety $\Sigma_{(n-1)^d}$ has dimension $d(n-1)$ and degree $\binom{d(n-1)}{n-1,n-1,\cdots,n-1}=\frac{(d(n-1))!}{d(n-1)!}$, a general subspace of $\PG(N-1,q^t)$ of codimension $d(n-1)$ contains at most $\frac{(d(n-1))!}{d(n-1)!}$ points of $\mathcal{V}_{d,\bsigma}$.\\
Moreover, the Segre variety is smooth and hence the tangent space $T_P(\Sigma_{(n-1)^d})$ to $\Sigma_{(n-1)^d}$ at a point $P= \langle v_0\otimes v_1\otimes \cdots \otimes v_{d-1} \rangle$ has dimension $d(n-1)$ and it spanned by the $d$ subspaces
$$\langle \langle v_0 \otimes v_1 \otimes \cdots v_{i-1} \otimes u_{i} \otimes v_{i+1} \otimes \cdots \otimes v_{d-1} \rangle \,|\,
\langle u_i\rangle \in \PG(n-1,q^t)\rangle \cong \PG(n-1,q^t).$$
These subspaces pairwise intersect only in $P$ and they are the maximal subspaces contained in $\Sigma_{(n-1)^d}$ through the point $P$, and  $\Sigma_{(n-1)^d}$ does not share with $\cV_{d,\bsigma}$ the property proved in Corollary \ref{indep2}. We have, in fact, $T_P(\Sigma_{(n-1)^d})\cap \cV_{d,\bsigma}=P$ for each $P \in \cV_{d,\bsigma}$.

\section{The code $\mathcal{C}_{d,\bsigma}$}
As we have seen in  Example \ref{example2}, the SLP-variety turns out to be a variety of a subgeometry of order $q$, even though the array $\bsigma$ is defined on a finite field of order $q^t$, hence among all the possible choice of $\bsigma$ and $n$, for $q$ 'big enough'  $\cV_{t,\sigma}$ is the variety with the most 'dense' set of points of a projective space with the property that any $d+1$ points are independent. In this case, since $d=t$ and, as proved in \cite{giuzzipepe13}, $t+2$ linearly dependent points are contained in a normal rational curve of degree $t$ of $\PG(t,q)$, $q>t$.\\
For the classical Veronese variety of degree $d$, hence for $\bsigma=\textbf{1}$, Theorem \ref{r+2} implies that $d+2$ linearly dependent points are contained in the Veronese embedding of degree $d$ of a line, hence in a normal rational curve of degree $d$ of $\PG(d,q^t)$.\\
Finally, for a general $(d,\bsigma)$-Veronese variety, if $d+2>q'+1$, with $q'$ defined as in Corollary \ref{r+2}, every $d+2$ points of   $\mathcal{V}_{d,\bsigma}$ are linearly independent, hence, for 'small' $q'$, it provides a dense set of points with that property. More precisely, we get $\frac{q^{nt}-1}{q^t-1}$ points in $\PG(N-1,q^t)$ such that any $d+2$ of them are in general position.  Sets of points with properties of this sort are studied for their connections with linear codes.\\

\noindent If $H$ is the matrix whose columns are the coordinates vectors of the points of the variety $\cV_{d,\bsigma}$, we get a code $\cC_{d,\bsigma}$ and we may study the minimum distance of it and characterize the codewords of minimum weight (for an overview on this topic, see, e.g., \cite{berlekamp}).

\begin{definition}\label{veronesiancode}
Let $\cV_{d,\bsigma}$  be a $(d,\bsigma)$-Veronese variety and denote by $\mathcal{C}_{d,\bsigma}$
the code whose parity check matrix $H$ of order $N\times (\frac{q^{nt}-1}{q^t-1})$ has  columns that are the coordinate vectors of the points
of the variety $\cV_{d,\bsigma}$.
\end{definition}

Clearly,  the order of the columns of $H$ is arbitrary, so that Definition \ref{veronesiancode} makes sense only up to code equivalence, as a permutation of the columns that is not usually an automorphism of the code, see \cite[Remark 3.3]{giuzzipepe13}.

\begin{definition}\label{support}
The support of a codeword $ \textbf{w} \in \cC_{d,\bsigma}$ is the set of the points of the variety $\cV_{d,\bsigma}$ corresponding to the non-zero positions of $\textbf{w}$.
\end{definition}

As showed in \cite[Theorem 3.5]{giuzzipepe13}, the following result holds

\begin{theorem}
Let $\bsigma \in G^d$ with $d_{\bsigma}=(d_0,d_1,\ldots,d_m)$, $|\boldsymbol{\sigma}|<q^t$ and $\F_{q'}$ be the largest  subfield fixed by $\sigma_i$'s. If $d<q'$ then the code $\cC_{d,\bsigma}$ has length $r=\frac{q^{nt}-1}{q^t-1}$ and parameters $[r,r-N,d+2]$.
\end{theorem}
\begin{proof}
Since $|\cV_{d,\bsigma}|=|\PG (n - 1, q^t )|$ the code $\cC_{d,\bsigma}$ has lenght $\frac{q^{nt}-1}{q^t-1}$ .
Moreover, since  $\cV_{d,\bsigma}$ is not contained in any hyperplane of $\PG(N-1,q^t)$,  the vector space dimension of the $N \times r$  matrix
$H$ is maximal and so the dimension of the code is $r - N$.  By Corollary \ref{indep2} guarantees
that any $ d + 1$ columns of $H$ are linearly independent; thus, by \cite[Theorem 10, p. 33]{MC}, the
minimum distance of $\cC_{d,\bsigma}$ is at least $d + 2$.
The image under $\nu_{d,\bsigma}$ of the canonical subline $\PG (1, q')$ of $\PG (n - 1, q^t)$ determines a
submatrix $H'$ of $H$ with many repeated rows; indeed, the points represented in $H$ constitute
a normal rational curve $\PG(d,q')$ and  it follows that any $d + 2$ such points are necessarily dependent. Hence, the minimum distance
is exactly $d+ 2$.
\end{proof}

Now, as  in \cite[Theorem 3.7]{giuzzipepe13}, by the characterizations of sets of $d+2$ points of $\cV_{d,\bsigma}$ which are linearly dependent yields a characterization of the minimum weight codewords of the associated code. More precisely,

\begin{theorem}
A codeword $\textnormal{\textbf{w}} \in \cC_{d,\bsigma}$ has minimum weight if and only if its support consists of $d+2$ points contained in the image of a subline $\PG(1,q')$, $d<q'$, where $\F_{q'}$ is the largest subfield of $\F_{q^t}$ fixed by $\sigma_i$ for all $\sigma_i$ in $\bsigma$.
\end{theorem}

Suppose $d \geq q'$ where $\mathbb{F}_{q'}$ is the largest subfield of $\mathbb{F}_{q^t}$ fixed by $\sigma_i$, for all $\sigma_i$ in $\bsigma$. By Theorem \ref{r+2}, the code $\cC_{d,\bsigma}$ is a linear code with minimum distance $ d+3 \leq \delta \leq N+1$. If the Singleton bound is reached, then it is an MDS code.
Let $N$ be  as in \eqref{N} with $\displaystyle\sum_{i=0}^{m}d_i=d$.
If $n=2$, then $$N=\displaystyle\prod_{i=0}^{m}(d_i+1)$$
and the minimum is reached for $m=1,d_0=d-1,d_1=1$, so $N=2d$.\\
If $\sigma$ is such that $\Fix(\sigma)\cap \mathbb{F}_{q^t}=\mathbb{F}_p$, where $p$ is the characteristic of the field, since we should have $d\geq p$, the smallest possible $d=p$ and in this case
\begin{equation}\label{array}
\bsigma=(\underbrace{1,1,\ldots,1}_{p-1\, \textnormal{times}},\sigma)
\end{equation}
getting that $\mathcal{V}_{d,\bsigma}$ is a set of $q^t+1$ points in $\PG(2p-1,q^t)$ such that any $p+2$ of them are in general position. So the code $\cC_{d,\bsigma}$ is a $[q^t+1,q^t-2p+1]$-linear code with minimum distance at least $p+3$ and the Singleton bound $2p+1$.  Now, if $\sigma: x \mapsto x^p$, then $\cV_{p,\bsigma}$  is the normal rational curve of $\PG(2p-1,q^t)$;  hence $\cC_{p,\bsigma}$ is an MDS code.\\
Furthermore for $p\in \{2,3\}$, the following cases can also occur
\begin{enumerate}
\item [-] for $p=2$, $\sigma: x \mapsto x^{2^h}$, $1 < h < et$, $\cV_{2,\bsigma}$ is either the Segre arc or the normal rational curve (for $h=et-1$),
hence  $\cC_{2,\bsigma}$ is an MDS code.
\item[-] for $p=3$, $\sigma: x \mapsto x^{3^h}$, $1 < h < et$, $\cV_{3,\bsigma}$ is a $(3^{et}+1)$-track of $\PG(5,3^{et})$; hence   $\cC_{3,\bsigma}$ is a so called \textit{almost} $MDS$ code, \cite{deboer}, see next Theorem \ref{p=3}.
\end{enumerate}
 Clearly, as $p$ gets larger, the minimum distance gets smaller than the Singleton bound.
Before showing the announced result,  we recall the following theorem due to Thas \cite{thas} and of which Kaneta and  Maruta gives an elementary proof,

\begin{theorem}\label{kmt}
\cite[Theorem 1]{KanetaMaruta} In $\PG(r,q)$, $r\geq 2$ and $q$ odd, every $k$-arc with
$$q -\sqrt{q}/4 + r-1/4 \leq k \leq q+1$$ is contained in one and only one normal rational curve of the space. In particular, if $q > (4r-5)^2$, then every $(q+1)$-arc is a normal rational curve.
\end{theorem}

\begin{theorem}\label{p=3}
Let $q=3^e$ and $\sigma: x \in \F_{q^t}  \mapsto x^{3^h} \in \F_{q^t}$,  $1<h<et$, $\gcd(h,et)=1$ with $et>4$. Then $\cV_{3,\bsigma}$ with $\bsigma=(1,1,\sigma)$ is a  $(3^{et}+1)$-track of $\PG(5,3^{et})$ and $\cC_{3,\bsigma}$
is an  almost MDS.
\end{theorem}
\begin{proof} By the previous considerations, since the $[q^t+1,q^t-5]$-code $\cC_{d,\bsigma}$ has distance at least $6$, the result follows showing  the existence of 6 columns of $H$ linearly dependent or equivalently that there exists $6$ points linearly dependent of the set
\begin{equation*}
\cV_{3,\bsigma}=\{(1,z,z^2,z^{3^h},z^{3^h+1},z^{3^h+2}) \,:\, z \in \F_{q^t}\} \cup \{(0,0,0,0,0,1)\}.
\end{equation*}

Suppose that any 6 points of $\cV_{3,\bsigma}$ with $\bsigma=(1,1,\sigma)$ are linearly independent, hence  $\cV_{3,\bsigma}$ is an arc of $\PG(5,q^t)$. By Theorem \ref{kmt},  $\cV_{3,\bsigma}$ must be projectively equivalent to rational normal curve

\begin{equation*}
\{(1,y,y^2,y^{3},y^{4},y^{5}) \,:\, y \in \F_{q^t}\} \cup \{(0,0,0,0,0,1)\}.
\end{equation*}

Since the normal rational curve has a 3-transitive automorphisms group, we can always  assume that there is a collineation  of $\PG(5,q^t)$ fixing $(0,0,0,0,0,1)$ and  $(1,0,0,0,0,0)$. Moreover, w.l.o.g. we can assume that this collineation has the identity as companion automorphism.

Hence there must be $f_i(y)\in \F_{q^t}[y]$ of degree at most 5 and linearly independent such that
\begin{equation*}
(f_0(y),f_1(y), f_2(y),f_3(y),f_4(y),f_5(y))=(1,z,z^2,z^{3^h},z^{3^h+1},z^{3^h+2})
\end{equation*}

with $f_i(y)$ vanishing in 0 for $i \in \{1,2,3,4,5\}$ and $f_0(0)=1$ up to a nonzero scalar. So, $f_0(y)=1$ for all $y \in \F_{q^t}$ and  since $\deg f_0(y)\leq 5 < q^t$,  then $f_0(y)=1$. Note that $\deg f_i(y) \neq 0$ for $i=1,2,3,4,5$ and
$$f_2(y)=f_1(y)^2 \,\,\, \mod \,\,\, y^{q^t}-y,$$
but  $2\deg f_1(y)\leq 10 < q^t$, and hence $f_2(y)=f_1(y)^2$ and $\deg f_1(y)\leq 2$.
Similarly,
$$f_4(y)=f_1(y)^{3^h} \mod \,\,\,y^{q^t}-y,$$
but $3^h\deg f_1(y)\leq 3^h\cdot 2 < q^t$, so $f_4(y)=f_1(y)^{3^h}$ and $3^h \deg f_1(y) \leq 5$, obtaining $3^h \leq 5$, a contradiction.
\end{proof}

Actually, the result above holds for $q^{t}=27,81$ as well, this is verified by the software MAGMA, obtaining an infinite family of almost MDS codes or, equivalently, an infinite family of $(3^{et}+1)$-tracks of $\PG(5,3^{et})$ with $et>2$.

\section{Acknowledgement}
The authors thank Italian National Group for Algebraic
and Geometric Structures and their Applications (GNSAGA - INdAM) for having supported this research.

Data Deposition Information: No datasets have been used

\end{document}